\documentclass[11pt]{amsart}
\usepackage[utf8]{inputenc}

\usepackage{amsmath,amssymb,amsfonts,amsthm,amscd,amstext,amsxtra,amsopn,array,url,verbatim,mathrsfs,enumerate,anysize,enumitem,graphicx,xfrac,mathtools}

\usepackage[table,xcdraw]{xcolor}

\usepackage [english]{babel}
\usepackage [autostyle, english = american]{csquotes}
\MakeOuterQuote{"}

\makeatletter
\def\imod#1{\allowbreak\mkern3mu({\operator@font mod}\,\,#1)}
\makeatother

\newtheorem{thm}{Theorem}[section]
\newtheorem{theorem}[thm]{Theorem}
\newtheorem{proposition}[thm]{Proposition}

\newtheorem{lemma}[thm]{Lemma}

\newtheorem*{theorem*}{Theorem}
\newtheorem*{problem*}{Problem}
\newtheorem*{corollary*}{Corollary}

\theoremstyle{definition}

\title{Explicit Subconvexity Estimates for Dirichlet $L$-functions}
\author{Forrest J. Francis}
\date{2022}

\begin{document}

\begin{abstract}
Given a Dirichlet character $\chi$ modulo a prime $q$ and its associated $L$-function, $L(s,\chi)$, we provide an explicit version of Burgess' estimate for $|L(s, \chi)|$. We use partial summation to provide bounds along the vertical lines $\Re{s} = 1 - {r}^{-1}$, where $r$ is a parameter associated with Burgess' character sum estimate. These bounds are then connected across the critical strip using the Phragm\'en--Lindel\"of principle. In particular, for $\sigma \in \left[\frac{1}{2}, \frac{9}{10}\right]$, we establish 

\begin{equation*}
    \left|L\left(\sigma + it, \chi\right)\right| \leq (1.105) (0.692)^\sigma   q^{\frac{31}{80}-\frac{2}{5}\sigma}(\log{q})^{\frac{33}{16}-\frac{9}{8}\sigma} |\sigma + it|.
\end{equation*}

\end{abstract}

\maketitle

\section{Introduction}

Consider a Dirichlet character $\chi$ modulo $q$ and its associated $L$-function, $L(s,\chi)$. It is regularly of interest for one to know about the size of $L(s, \chi)$ inside the critical strip ($0 < \Re{s} < 1$). A classical estimate is due to Davenport \cite{davenport1931}
\[L(\sigma+it, \chi) \ll (q\tau)^{\frac{1}{2}(1-\sigma)},\]
where $\tau = |t| + 1$.
Most authors focus on bounds along the critical line. Hiary \cite{hiary2016} shows that if one wishes for an explicit version of the trivial bound,  
\[|L(1/2+it,\chi)| \leq 124.46(q|t|)^\frac{1}{4} \hspace{20pt} (q|t| \geq 10^9\mbox{, }|t| \geq \sqrt{q}),\]
suffices for primitive characters $\chi$. 
However, the trivial bound is not a \emph{sub-convexity} bound on the half-line, that is one that shows
\[L(1/2+it, \chi) \ll q^{\theta + \epsilon} \tau,\]
with $\theta < \frac{1}{4}$.
We do have sub-convexity for Dirichlet $L$-functions, via Burgess \cite{burgess1963}, who showed that 
\begin{equation}\label{burgessL}
L(1/2 + it, \chi) \ll q^{\frac{3}{16}+\varepsilon}\tau.
\end{equation}
Petrow and Young \cite{petrow2020} showed that for $\chi$ of cube-free conductor,
\[L(1/2 + it, \chi) \ll_\epsilon q^{\frac{1}{6}+\epsilon} (1+|t|)^{\frac{1}{6}+\epsilon}.\]
For characters whose modulus is a very large prime power, we have estimates which beat the Weyl exponent of $\tfrac{1}{6}$, due to \cite{milicevic2016}. 
\begin{theorem}[Theorem 1, \cite{milicevic2016}]
Let $\theta > \theta_0 = 0.1645$. There is an $r > 0$ for which 
\[L(1/2, \chi) \ll p^r q^\theta \left(\log{q}\right)^\frac{1}{2},\]
holds for every Dirichlet character of modulus $q = p^n$.
\end{theorem}

The estimate  due to Burgess comes to us as a consequence of a character sum estimate known as the Burgess bound, which improves upon the P\'olya--Vinogradov inequality for certain lengths of character sums. Recently, in \cite{francis2021}, the author improved upon work of Trevi\~no which established explicit Burgess bounds for characters of prime modulus \cite{trevino2012}. For characters of composite modulus, further results can be found in \cite{jainsharma2021}.
\begin{theorem}[Theorem 1.8, \cite{francis2021}]\label{burgessbound}
Let $p$ be a prime and $2 \leq r \leq 10$ be an integer. Let $\chi$ be a non-principal character modulo $p$. Let $N_0$ and $N_1$ be non-negative integers with the condition that $1 \leq N_1 < 2p^\frac{5}{8}$ only when $r = 2$. Then, for $p \geq 10^{10}$, we have a precisely determined constant $B(r)$ for which
\begin{equation}\label{burgess}
\left| \sum_{n = N_0 +1}^{N_0+N_1}\chi(n) \right| < B(r)N_1^{1-{r}^{-1}}p^{\frac{r+1}{4r^2}}\left(\log{p}\right)^\frac{1}{2r}.
\end{equation}
\end{theorem}
For example, when $r=3$, one can take $B(3) = 2.491$. An extensive table of constants is provided in \cite[Table 3]{francis2021}. 

In this article, we are interested in using the explicit estimates of Theorem \ref{burgessbound} to obtain explicit bounds on $L(\sigma +it,\chi)$ for any $0< \sigma < 1$. Hiary \cite{hiary2016} uses P\'olya--Vinogradov along with partial summation to establish the explicit bound
\begin{equation}\label{nonconvex}
    \lvert L(1/2+it, \chi) \rvert \leq 4 q^{\frac{1}{4}} \sqrt{\tau\log{q}}.
\end{equation}
As Hiary notes, \eqref{nonconvex} is useful when $\tau$ is small (relative to $q$). On the other hand, one expects explicit bounds for $|L(1/2 +it, \chi)|$ which come from \eqref{burgessL} to be useful when $\tau$ is very small, roughly $\tau \ll q^{\frac{1}{8}+\varepsilon}$. 

We will use the approach of Burgess to obtain bounds of the same type as \eqref{burgessL} on the whole critical strip. For each choice of $r$ in the Burgess bound, we obtain a bound for the size of $L(\frac{r-1}{r} + it)$, and then use the functional equation and the Phragm\'en--Lindel\"of principle to complete the bound for any $\sigma \in \left[\frac{1}{10},\frac{9}{10}\right]$. This yields the following.

\begin{theorem}\label{main}
Let $q$ be a prime and $\chi$ be a non-principal, primitive character modulo $q \geq 10^{10}$ and $|t| \geq 1$. Then

\begin{equation*}
    \left| L\left(\frac{1}{2}+ it, \chi\right) \right| \leq 0.918 q^\frac{3}{16} (\log{q})^\frac{3}{2}  \left|\frac{1}{2}+it\right|.
\end{equation*}
Moreover, for $\sigma \in \left[\frac{1}{2}, \frac{9}{10}\right]$,
\begin{equation*}
    \left|L\left(\sigma + it, \chi\right)\right| \leq (1.105) (0.692)^\sigma   q^{\frac{31}{80}-\frac{2}{5}\sigma}(\log{q})^{\frac{33}{16}-\frac{9}{8}\sigma} |\sigma + it|.
\end{equation*}
Finally, for $\sigma \in \left[\frac{1}{10}, \frac{1}{2}\right]$
\begin{equation*}
    \left|L\left(\sigma + it, \chi\right)\right| \leq (10.094)(0.083)^\sigma q^{\frac{1}{2}-\frac{5}{8}\sigma}(\log{q})^{\frac{7}{4}-\frac{\sigma}{2}} |\sigma + it|.
\end{equation*}

\end{theorem}

Additionally, one may use these methods to obtain bounds for $|L(1 + it,\chi)|$, although such bounds are likely only useful for very specific ranges of $|t|$. 

\begin{theorem}\label{oneline}
Let $q$ be a prime and $\chi$ be a non-principal, primitive character with modulus $q$. Then, for $|t| \geq 1$, $\epsilon > 0$, and $\sigma \in \left[\frac{9}{10}, 1+\epsilon \right]$, we have

\begin{equation*}
    \begin{aligned}
    |L(\sigma +it,\chi)| &\leq \left(0.792 \left|\frac{9}{10}+it\right| q^{\frac{11}{400}} (\log{q})^{\frac{21}{20}}\right)^\frac{1+\epsilon-\sigma}{\frac{1}{10}+\epsilon} \\ &\cdot
   \left( 1 + \epsilon^{-1}\right)^\frac{\sigma - \frac{9}{10}}{\frac{1}{10}+\epsilon}.
    \end{aligned}
\end{equation*}

Specifically, for $\sigma = 1$, we have

\[|L(1 +it,\chi)| \leq \left(1+\epsilon^{-1}\right)^\frac{1}{1+10\epsilon} \left(0.792q^{\frac{11}{400}} (\log{q})^{\frac{21}{20}}\right)^\frac{10\epsilon}{1 + 10\epsilon}\left|\frac{9}{10}+it\right|^\frac{10\epsilon}{1+10\epsilon}.\]
\end{theorem}

In the future, it is certainly worth investigating explicit versions of the theorems of Heath--Brown in \cite{heathbrown1978}, an example of which includes, for a specific set of moduli $q$,$L(1/2 + it, \chi) \ll (q\tau)^{\frac{1}{6}+\epsilon}$.

Explicitly, Hiary \cite[Corollary 1.1]{hiary2016} provides
\[|L(1/2+it,\chi)| \leq 9.05 d(q) (q|t|)^\frac{1}{6} \log^{\frac{3}{2}}{q|t|}\]
with $|t| > 200$, $\chi$ a character modulo a sixth power $q$, and $d(q)$ as the number of divisors function. 

However, it appears that making these results explicit would require more than the elementary means of this article.

\section{General Results and Useful Lemmas}

In the course of the proof of Theorem \ref{main}, we will obtain bounds for $L(s, \chi)$ depending on the parameter $r$ arising in Burgess' bound and $\sigma = \Re{s}$. Roughly speaking, we will follow the approaches of Burgess \cite{burgess1962} and Hiary \cite[Eq. 4]{hiary2016}, and "only" use partial summation. As a first step in this direction, we make \cite[Lemma 10]{burgess1962} explicit. Doing so requires an explicit P\'olya--Vinogradov inequality, i.e., 
\[\left|\sum_{n=1}^{N} \chi(n)\right| \leq P\sqrt{q}\log{q},\]
where $P$ is an explicit constant.

\begin{lemma}\label{partsum}
Let $\chi$ be a primitive character modulo $q>1$ and $L(s,\chi)$ be the associated Dirichlet $L$-function. Then, for $s= \sigma +it$ with $\sigma > 0$ fixed, we have
\begin{equation}
\begin{aligned}|L(\sigma+it, \chi)| &\leq \frac{M^{1-\sigma}}{|1-\sigma|} + \frac{1}{M^\sigma}\left|\sum_{n\leq M}\chi(n)\right| + (\sigma+|t|)P\sqrt{q}\log{q}\frac{N^{-\sigma}}{\sigma} \\&+ (\sigma +|t|)\int_M^{N} \! \left| \sum_{1\leq n \leq u} \chi(n) \right| u^{-1-\sigma} \, \textrm{d}u,
\end{aligned}
\end{equation}
where $M, N$ are arbitrary constants satisfying $0< M \leq N$.
\end{lemma}

\begin{proof}
Since $\chi$ is non-principal, for a fixed $\sigma > 0$, we may write
\begin{equation}\label{split}
L(\sigma+it, \chi) = \sum_{n = 1}^{M} \frac{\chi(n)}{n^{\sigma+it}} + \sum_{n  = M +1}^{\infty}\frac{\chi(n)}{n^{\sigma+it}}.
\end{equation}

Estimate the initial sum trivially. To estimate the tail, cut it at some point $N$ and use partial summation to obtain:

\begin{equation}\label{sumbound}
        \sum_{M< n \leq N} \frac{\chi(n)}{n^{\sigma+it}} = \frac{1}{N^s} \sum_{n\leq N} \chi(n) +  \frac{1}{M^s} \sum_{n\leq M} \chi(n) + s\int_M^{N} \! \sum_{1\leq n \leq u} \chi(n) u^{-1-s} \, \textrm{d}u.
\end{equation}

On the other hand,
\begin{equation}\label{endbound}
        \left| \sum_{N< n \leq \infty} \frac{\chi(n)}{n^{\sigma+it}} \right| \leq  \frac{1}{N^s} \sum_{n\leq N} \chi(n) + (\sigma + |t|)P\sqrt{q}\log{q}\frac{N^{-\sigma}}{\sigma}.
\end{equation}
Combining \eqref{sumbound} and \eqref{endbound} with the trivial estimate and taking absolute values gives the result.
\end{proof}

Taking Burgess' character sum estimate with parameter $r$ in Lemma 2.1 offers a bound on $|L(s,\chi)|$ which depends on $r$. The shape of this bound depends on the relationship between $r$ and $\sigma$, as detailed in the following proposition. As detailed in Proposition \ref{offbound}, one should observe that when $\sigma > 1- {r}^{-1}$, we have $\sigma r - r +1 >0$. Therefore, there the second line of \eqref{offcomplete} may be ignored with regard to asymptotics. A clever choice of parameters may even eliminate these terms entirely. The same is true for the first line of \eqref{offcomplete} when $\sigma < 1 - {r}^{-1}$. In this sense, Proposition \ref{offbound} should be compared with Theorem 3 of \cite{burgess1963}. Note that the distinction between cases below is entirely due to the behaviour of the integral in Lemma \ref{partsum}.

\begin{proposition}\label{offbound}
Suppose $\chi$ is a primitive character modulo $q$ for which constants $B, r$ exist such that
\begin{equation*}
\left| \sum_{n = N_0 +1}^{N_0+N_1}\chi(n) \right| < BN_1^{1-{r}^{-1}}q^{\frac{r+1}{4r^2}}\left(\log{q}\right)^\frac{1}{2r},
\end{equation*}
for $N_1$ at least as large as $q^\frac{2r+1}{4r}(\log{q})^\frac{2r-1}{2r-2}$.
Let $\sigma \in (0,1)$. Then, for $\sigma \neq 1 - \tfrac{1}{r}$, we have
\begin{equation}\label{offcomplete}
    \begin{aligned}
    \left| L(\sigma + it,\chi) \right| &\leq \left( \frac{1}{1-\sigma} + B + B(\sigma +|t|)\left(\frac{r}{\sigma r -r + 1}\right)\right) q^{\frac{r+1}{4r} - \sigma \frac{r+1}{4r}} (\log{q})^\frac{1-\sigma}{2} \\ &+ \left(P\frac{(\sigma + |t|)}{\sigma} - B(\sigma+|t|)\left(\frac{r}{\sigma r - r+1}\right) \right)q^{\frac{1}{2} - \frac{\sigma(2r+1)}{4r}} (\log{q})^{1-\frac{\sigma(2r-1)}{2r-2}},
    \end{aligned}
\end{equation}
whereas, for $\sigma = 1 - \tfrac{1}{r}$, we have
\begin{equation}\label{eqcomplete}
    \begin{aligned}
    \left| L(\sigma + it,\chi) \right| &\leq \left(\frac{1}{1-\sigma} + B + P\frac{(\sigma +|t|)}{\sigma} + (\sigma + |t|)B\left(\frac{\log{q}}{4} + \frac{\log{\log{q}}}{2\sigma}\right)  \right)\\ &\cdot q^{\frac{1}{2}-\frac{\sigma}2+\frac{\sigma^2}{4}}(\log{q})^\frac{1-\sigma}{2} 
    \end{aligned}
    \end{equation}
\end{proposition}

\begin{proof}
For such a $\chi$, when $\sigma + {r}^{-1} \neq 1$, Lemma \ref{partsum} provides

\begin{equation}
\begin{aligned}|L(\sigma+it, \chi)| &\leq \frac{M^{1-\sigma}}{1-\sigma} + BM^{1-\sigma -{r}^{-1}}q^\frac{r+1}{4r^2}\log{q}^\frac{1}{2r} + (\sigma+|t|)P\sqrt{q}\log{q}\frac{N^{-\sigma}}{\sigma} \\&+ (\sigma +|t|)Bq^\frac{r+1}{4r^2}(\log{q})^\frac{1}{2r}\left(\frac{r}{\sigma r - r + 1}\right)\left(M^{1-\sigma -{r}^{-1}}- N^{1-\sigma-{r}^{-1}}\right).
\end{aligned}
\end{equation}
Making the choices $M = q^{\frac{r+1}{2r(r+2)}}(\log{q})^\frac{1}{2}$ and $N = q^\frac{2r+1}{4r}(\log{q})^\frac{2r-1}{2r-2}$ yields \eqref{eqcomplete}.

On the other hand, if $\sigma + {r}^{-1} = 1$, then Lemma \ref{partsum} provides

\begin{equation}
\begin{aligned}|L(\sigma+it, \chi)| &\leq \frac{M^{1-\sigma}}{1-\sigma} + BM^{1-\sigma -{r}^{-1}}q^\frac{r+1}{4r^2}\log{q}^\frac{1}{2r} + (\sigma+|t|)P\sqrt{q}\log{q}\frac{N^{-\sigma}}{\sigma} \\&+ (\sigma +|t|)Bq^\frac{r+1}{4r^2}(\log{q})^\frac{1}{2r}\left(\log{N}-\log{M}\right).
\end{aligned}
\end{equation}
Here, the choices $M = q^\frac{2-\sigma}{4}(\log{q})^\frac{1}{2}$  and $N = q^\frac{3-\sigma}{4}(\log{q})^\frac{\sigma +1}{2\sigma}$ yield \eqref{eqcomplete}.
\end{proof}

While \eqref{offcomplete} agrees with Theorem 3 of \cite{burgess1963} (except in the case $\sigma = \frac{1}{2}$, where $\eqref{eqcomplete}$ fills in), \eqref{eqcomplete} tends to provide better exponents on $q$. We can now easily obtain bounds along the lines $\sigma = \frac{2}{3}, \frac{3}{4}, \ldots, \frac{r-1}{r} \ldots$, where $r>2$ is an integer. The bounds we obtain along vertical lines can be connected to bound $L(s,\chi)$ across the entire strip. For this, we use the following version of the Phragm\'en--Lindel\"of principle, which comes from Rademacher \cite{rademacher1959}. 

\begin{theorem}[Theorem 2, \cite{rademacher1959}]\label{rademacher}
Let $f(s)$ be regular analytic in the strip $a \leq \Re{s} \leq b$. Suppose for some positive constants $c$, $C$, we have
\[|f(s) | < C e^{|t|^c}.\]
Additionally, suppose we have
\[|f(a+it)| \leq A|Q+ a +it|^\alpha, \]
and
\[|f(b+it)| \leq B|Q+ b +it|^\beta,\]
where
\[Q + a > 0 \mbox{ and } \alpha \geq \beta.\]
Then, in the strip $a \leq \Re{s} \leq b$, we have
\[ |f(s)| \leq \left(A|Q+s|^\alpha\right)^{\frac{b-\sigma}{b-a}} \cdot \left(B|Q+s|^\beta\right)^{\frac{\sigma - a}{b-a}}.  \]
\end{theorem}

\section{Proof of Theorem \ref{main}}

First, let us consider $\sigma = \frac{1}{2}$ and $r = 2$. We do not have access to a bound of the form \eqref{eqcomplete}, since Theorem \ref{burgessbound} is restricted when $r=2$. We can get around this by sacrificing the exponent on the logarithm in \eqref{eqcomplete} and using a Burgess bound in the form 
\begin{equation}\label{BB2}
\left| \sum_{n = N_0 +1}^{N_0+N_1}\chi(n) \right| < B(2)N_1^{\frac{1}{2}}q^{\frac{3}{16}}\left(\log{q}\right)^\frac{1}{2}. 
\end{equation}
From Table 6 in \cite{francis2021}, the above bound holds with $B(2) = 1.520$ when we have a prime $q > 10^{10}$. Arguing as in Proposition \ref{offbound}, we may replace \eqref{eqcomplete} with 
\begin{equation}\label{2complete}
    \begin{aligned}
    \left| L\left(\frac{1}{2} + it,\chi\right) \right| &\leq \left(\frac{1}{2} + |t|\right)q^{\frac{3}{16}}(\log{q})^\frac{3}{2} \\ &\cdot \left( B(2)\left(\frac{1}{4} + \frac{2\log\log{q}}{\log{q}}\right) + \frac{2+B(2)}{(\sigma +|t|)\log{q}}+ \frac{2P}{(\log{q})^2}\right) \\ &\leq 0.918 \left(\frac{1}{2} + |t|\right)q^{\frac{3}{16}}(\log{q})^\frac{3}{2} \,.
    \end{aligned}
    \end{equation}
To obtain the constant, we have assumed that $q \geq 10^{10}$ and $t\geq 1$ and taken $P = {2\pi^{-2}}+ (\log{10^{10}})^{-1}$, as in \cite{frolenkov2013}.
For other integers $r$, we may appeal to Proposition \ref{offbound} via Theorem \ref{burgessbound} directly, obtaining bounds of the form
\begin{equation}\label{genericverts}
    \left| L\left(1-{r}^{-1} + it, \chi\right) \right| \leq C_{1-{r}^{-1}} \left(1 - {r}^{-1} + |t|\right) q^{\beta_r} (\log{q})^{\gamma_r} \,.
\end{equation}

Table \ref{table:verticals} records specific values for each $r$ from 2 to 10. Here, for $r \geq 3$, we have used the constants $B(r)$ from Table 3 of \cite{francis2021}. For $r$ larger than 10, one should in principle be able to take $B(10)$ in place of determining the explicit Burgess constant directly. The general form of $\beta_r$ is $\frac{r+1}{4r^2}$, while $\gamma_r$ is $\tfrac{2r+1}{2r}$ except when $r=2$.

\begin{table}[ht]
\centering
\caption{Acceptable values in \eqref{genericverts} for several $r$.}
\label{table:verticals}
\begin{tabular}{|r|r|r|r|r|r|}
\hline
\rowcolor[HTML]{C0C0C0} 
\multicolumn{1}{|c|}{\cellcolor[HTML]{C0C0C0}$r$} & \multicolumn{1}{c|}{\cellcolor[HTML]{C0C0C0}$\sigma$} & \multicolumn{1}{c|}{\cellcolor[HTML]{C0C0C0}$B(r)$} & \multicolumn{1}{c|}{\cellcolor[HTML]{C0C0C0}$C_{1-{r}^{-1}}$} & \multicolumn{1}{c|}{\cellcolor[HTML]{C0C0C0}$\beta_r$} & \multicolumn{1}{c|}{\cellcolor[HTML]{C0C0C0}$\gamma_r$} \\ \hline
2                                                 & $\sfrac{1}{2}$                                         & 1.5197                                              & 0.918                                               & $\sfrac{3}{16}$                                         & $\sfrac{3}{2}$                                           \\
3                                                 & $\sfrac{2}{3}$                                         & 2.4910                                              & 1.036                                               & $\sfrac{1}{9}$                                          & $\sfrac{7}{6}$                                           \\
4                                                 & $\sfrac{3}{4}$                                         & 2.1551                                              & 0.902                                               & $\sfrac{5}{64}$                                         & $\sfrac{9}{8}$                                           \\
5                                                 & $\sfrac{4}{5}$                                         & 1.9688                                              & 0.842                                               & $\sfrac{3}{50}$                                         & $\sfrac{11}{10}$                                         \\
6                                                 & $\sfrac{5}{6}$                                         & 1.8476                                              & 0.812                                               & $\sfrac{7}{144}$                                        & $\sfrac{13}{12}$                                         \\
7                                                 & $\sfrac{6}{7}$                                         & 1.7596                                              & 0.798                                               & $\sfrac{2}{49}$                                         & $\sfrac{15}{14}$                                         \\
8                                                 & $\sfrac{7}{8}$                                         & 1.6875                                              & 0.791                                               & $\sfrac{9}{256}$                                        & $\sfrac{17}{16}$                                         \\
9                                                 & $\sfrac{8}{9}$                                         & 1.6292                                              & 0.789                                               & $\sfrac{5}{162}$                                        & $\sfrac{19}{18}$                                         \\
10                                                & $\sfrac{9}{10}$                                        & 1.5810                                              & 0.792                                               & $\sfrac{11}{400}$                                       & $\sfrac{21}{20}$                                         \\ \hline
\end{tabular}
\end{table}

For $\sigma < \tfrac{1}{2}$, observe that Proposition \ref{offbound} taken with $r = 2$ (adjusting the exponent on the logarithms as necessary) yields

\begin{equation*}
\begin{aligned}
|L(\sigma + it, \chi)| &\leq (\sigma+|t|)q^{\frac{4-5\sigma}{8}}(\log{q})^{2-3\sigma} \\ 
&\cdot\left(\left(\frac{\frac{1}{1-\sigma}+B(2)}{\sigma +|t|} +\frac{2B(2)}{2\sigma-1}\right)q^{\frac{2\sigma}{8}-\frac{1}{8}}(\log{q})^{2\sigma -1} + \left(\frac{P}{\sigma} - \frac{2B(2)}{2\sigma-1} \right)\right).
\end{aligned}
\end{equation*}
We may bound the above by again assuming that $|t|\geq 1$. This gives us bounds of the form 
\begin{equation}\label{sigmacbounds}
    |L(\sigma + it, \chi)| \leq C_\sigma(\sigma+|t|)q^{\beta_\sigma}(\log{q})^\gamma_\sigma \,.
\end{equation}
We have recorded the value of $C_\sigma$ for several $\sigma$ in Table \ref{table:sigmaverticals}.

\begin{table}[ht]
\centering
\caption{The value $C_\sigma$ in \eqref{sigmacbounds} for $\sigma = \tfrac{1}{r}$.}
\label{table:sigmaverticals}
\begin{tabular}{|r|r|r|r|}
\hline
\rowcolor[HTML]{C0C0C0} 
\multicolumn{1}{|c|}{\cellcolor[HTML]{C0C0C0}$\sigma$} & \multicolumn{1}{c|}{\cellcolor[HTML]{C0C0C0}$C_\sigma$} & \multicolumn{1}{c|}{\cellcolor[HTML]{C0C0C0}$\beta_\sigma$} & \multicolumn{1}{c|}{\cellcolor[HTML]{C0C0C0}$\gamma_\sigma$} \\ \hline
$\sfrac{1}{3}$                                         & 8.934                                                   & $\sfrac{7}{24}$                                             & $1$                                                          \\
$\sfrac{1}{4}$                                         & 6.877                                                   & $\sfrac{11}{32}$                                            & $\sfrac{5}{4}$                                               \\
$\sfrac{1}{5}$                                         & 6.222                                                   & $\sfrac{3}{8}$                                              & $\sfrac{7}{5}$                                               \\
$\sfrac{1}{6}$                                         & 5.996                                                   & $\sfrac{19}{48}$                                            & $\sfrac{3}{2}$                                               \\
$\sfrac{1}{7}$                                         & 5.953                                                   & $\sfrac{23}{56}$                                            & $\sfrac{11}{7}$                                              \\
$\sfrac{1}{8}$                                         & 6.003                                                   & $\sfrac{27}{64}$                                            & $\sfrac{13}{8}$                                              \\
$\sfrac{1}{9}$                                         & 6.109                                                   & $\sfrac{31}{72}$                                            & $\sfrac{5}{3}$                                               \\
$\sfrac{1}{10}$                                        & 6.249                                                   & $\sfrac{7}{16}$                                             & $\sfrac{17}{10}$ \\
\hline
\end{tabular}
\end{table}

For $\sigma$ between the vertical lines $\Re{s} =1 - \tfrac{1}{r}$ and $\Re{s} = 1- \tfrac{1}{R}$ ($r < R$), we can apply Theorem \ref{rademacher}. Unless $r$ = 2, this gives us the bound

\begin{equation}\label{rademachergeneral}
\begin{aligned}
    |L(\sigma+it,\chi)| &\leq \left(C_{1-{r}^{-1}}  q^\frac{r+1}{4r^2} (\log{q})^\frac{2r+1}{2r}\left|\sigma+it\right|\right)^\frac{1-{R}^{-1} - \sigma}{{r}^{-1}-{R}^{-1}} \\
    &\cdot  \left(C_{1-{R}^{-1}}  q^\frac{R+1}{4R^2} (\log{q})^\frac{2R+1}{2R}\left|\sigma+it\right|\right)^\frac{\sigma -  \left(1-{r}^{-1}\right)}{{r}^{-1}-{R}^{-1}}.
\end{aligned}
\end{equation}
When $r =2$, the only change to \eqref{rademachergeneral} we should make is that the exponent on the first $\log{q}$ will be replaced by $\tfrac{3}{2}$. Taking $r = 2$ and $R = 10$, we obtain the second claim of Theorem \ref{main}. A similar approach using \eqref{sigmacbounds} with the line $\Re{s} = \frac{1}{10}$ and \eqref{2complete} completes Theorem \ref{main}.

These bounds may be extended to include $\sigma = 1$.
\begin{proof}[Proof of Theorem \ref{oneline}]
We require an estimate for the size of $L(s, \chi)$ for $s = 1+\epsilon + it$. One simple estimate is

\begin{equation}\label{easyvert}
    | L(1+\epsilon+it,\chi) | \leq \zeta(1+\epsilon) < 1 + \frac{1}{\epsilon}.
\end{equation}

Now, take Theorem \ref{rademacher} using \eqref{genericverts} (with $r=10$) and \eqref{easyvert}. This provides

\begin{equation}\label{aroundone}
    \begin{aligned}
    |L(\sigma +it,\chi)| &\leq \left(0.792 \left|\frac{9}{10}+it\right| q^{\frac{11}{400}} (\log{q})^{\frac{21}{20}}\right)^\frac{1+\epsilon-\sigma}{\frac{1}{10}+\epsilon} \\ &\cdot
   \left( 1 + \epsilon^{-1}\right)^\frac{\sigma - \frac{9}{10}}{\frac{1}{10}+\epsilon}
    \end{aligned}
\end{equation}
for $\frac{9}{10} \leq \sigma \leq 1 + \epsilon$. 
For $\sigma = 1$, this bound becomes

\[|L(\sigma +it,\chi)| \leq\left(1+\epsilon^{-1}\right)^\frac{1}{1+10\epsilon} \left(0.792q^{\frac{11}{400}} (\log{q})^{\frac{21}{20}}\right)^\frac{10\epsilon}{1 + 10\epsilon}\left|\frac{9}{10}+it\right|^\frac{10\epsilon}{1+10\epsilon}.\]

\end{proof}

\section*{Acknowledgments}
I would like to thank my supervisor, Tim Trudgian, for suggesting this topic and for his helpful discussions. 

\newpage
\bibliography{bibliographic.bib} 

\begin{thebibliography}{10}

\bibitem{burgess1962}
D.~A. Burgess.
\newblock On character sums and {$L$}-series.
\newblock {\em Proc. London Math. Soc. (3)}, 12:193--206, 1962.

\bibitem{burgess1963}
D.~A. Burgess.
\newblock On character sums and {$L$}-series. {II}.
\newblock {\em Proc. London Math. Soc. (3)}, 13:524--536, 1963.

\bibitem{davenport1931}
H.~Davenport.
\newblock On {D}irichlet's {$L$}-{F}unctions.
\newblock {\em J. London Math. Soc.}, 6(3):198--202, 1931.

\bibitem{francis2021}
F.~J. Francis.
\newblock An investigation into explicit versions of {B}urgess' bound.
\newblock {\em J. Number Theory}, 228:87--107, 2021.

\bibitem{frolenkov2013}
D.~A. Frolenkov and K.~Soundararajan.
\newblock A generalization of the {P}\'{o}lya-{V}inogradov inequality.
\newblock {\em Ramanujan J.}, 31(3):271--279, 2013.

\bibitem{heathbrown1978}
D.~R. Heath-Brown.
\newblock Hybrid bounds for {D}irichlet {$L$}-functions.
\newblock {\em Invent. Math.}, 47(2):149--170, 1978.

\bibitem{hiary2016}
G.~A. Hiary.
\newblock An explicit hybrid estimate for {$L(1/2+it,\chi)$}.
\newblock {\em Acta Arith.}, 176(3):211--239, 2016.

\bibitem{jainsharma2021}
N.~Jain-Sharma, T.~Khale, and M.~Liu.
\newblock Explicit {B}urgess bound for composite moduli.
\newblock {\em Int. J. Number Theory}, 17(10):2207--2219, 2021.

\bibitem{milicevic2016}
D.~Mili\'cevi\'c.
\newblock Sub-{W}eyl subconvexity for {D}irichlet {$L$}-functions to prime
  power moduli.
\newblock {\em Compos. Math.}, 152(4):825–875, 2016.

\bibitem{petrow2020}
I.~Petrow and M.~P. Young.
\newblock The {W}eyl bound for {D}irichlet {$L$}-functions of cube-free
  conductor.
\newblock {\em Ann. of Math. (2)}, 192(2):437--486, 2020.

\bibitem{rademacher1959}
H.~Rademacher.
\newblock On the {P}hragmén-{L}indelöf theorem and some applications.
\newblock {\em Math. Zeit.}, 72:192--204, 1959/60.

\bibitem{trevino2012}
E.~Trevi\~{n}o.
\newblock The least inert prime in a real quadratic field.
\newblock {\em Math. Comp.}, 81(279):1777--1797, 2012.

\end{thebibliography}
\bibliographystyle{plain}
\end{document}